\documentclass[hidelinks,11pt]{amsart}
\usepackage{amsmath,amsthm,amscd,amssymb}
\usepackage{geometry}
\usepackage{subcaption}
\usepackage{pb-diagram}
\usepackage{graphicx}
\usepackage{setspace}

\usepackage{amssymb}

\usepackage{pstricks}
\usepackage{pst-node}
\usepackage{url}
\usepackage{xcolor}
\usepackage{hyperref}

\usepackage{mathrsfs}
\usepackage{MnSymbol}
\usepackage{gensymb} 
\usepackage{enumerate}
\usepackage{mathtools} 
\usepackage[abs]{overpic}

\newtheorem{thm}{Theorem}[section]

\newtheorem{prop}[thm]{Proposition}
\newtheorem{cor}[thm]{Corollary}

\theoremstyle{definition}
\newtheorem{ex}[thm]{Example}

\newtheorem{rem}[thm]{Remark}

\newtheorem{conj}[thm]{Conjecture}
\numberwithin{equation}{section}
\numberwithin{figure}{section}

\newcommand{\Z}{\mathbb{Z}}

\newcommand{\R}{\mathbb{R}}

\newcommand{\xrk}{X_{r}(K)}
 
\newcommand{\rp}{\R \text{P}^{2}}

\DeclareMathOperator{\arf}{Arf}
\newcommand{\whd}{\text{Wh}^{+}_{t}(K)}
\newcommand{\pmwhd}{\text{Wh}^{\pm}_{t}(K)}

\usepackage{titlesec}
\titleformat{\section} {\normalfont\scshape\bfseries\filcenter}{\thesection }{1em}{}
\titleformat{\subsection} {\normalfont\scshape\bfseries\filcenter}{\thesubsection}{1em}{}

\begin{document}

\title[{Whitehead DOUBLES AND NON-ORIENTABLE SURFACES}]{Whitehead DOUBLES AND NON-ORIENTABLE SURFACES}

\author[MEGAN FAIRCHILD]{MEGAN FAIRCHILD}
\address{Department of Mathematics \\ Louisiana State University}
\email{mfarr17@lsu.edu}

\maketitle


\textbf{ABSTRACT.} Whitehead doubles provide a plethora of examples of knots that are topologically slice but not smoothly slice. We discuss the problem of the Whitehead double of the Figure 8 knot and survey commonly used techniques to obstructing sliceness. Additionally, we improve bounds in general for the non-orientable 4 genus of $t$-twisted Whitehead doubles and provide genus 1 non-orientable cobordisms to cable knots.

\section{INTRODUCTION}

Determining whether or not a knot is slice has been a problem of interest for a century now. The knots constructed from the Whitehead double operation present an intriguing phenomenon of topologically slice knots that are not smoothly slice, implying the existence of exotic structures in the fourth dimension. A knot of increasing interest is the positive untwisted Whitehead double of the figure 8 knot, which is topologically slice and believed to not be smoothly slice, yet a concrete proof has not been presented. One is not presented in this paper either, however a thorough review of all that has been attempted thus far as well as predictions and further directions will be discussed. We will also provide specific upper bounds on the non-orientable 4 genus of Whitehead doubles with the hopes that this may enlighten the effort to prove the positive untwisted Whitehead double of the figure 8 knot is not smoothly slice. 

A knot $K$ in $S^{3}$ is called \textit{smoothly slice} if it bounds a smoothly embedded disk in $B^{4}$, or its 4 genus is zero, $g_{4}(K) = 0$. For the topological category, a if a knot $K$ bounds a disk which has a topologically locally flat embedding in $B^{4}$, we call that knot \textit{topologically slice} and the notation is $g_{4}^{top}(K) =0$. We define non-orientable 4 genus in the spirit of Murakami and Yasuhara \cite{MY} as the minimum first Betti number amongst all non-orientable surfaces bounded by $K$, $\gamma_{4}(K) = \min \{ b_{1}(F) \hspace{1mm} | \hspace{1mm} F \text{ non-orientable, } \partial(F) = K, F \subset B^{4} \}$. We introduce the property \textit{non-orientable sliceness} as a knot bounding a M\"obius band, $\gamma_{4}(K) = 1$.

\begin{thm}\label{thm:NO4Gnotpreserved}
    The $t$-twisted Whitehead double operation does not preserve non-orientable sliceness. 
\end{thm}

We introduce Theorem \ref{thm:NO4Gnotpreserved} as it provides an interesting non-orientable analog to the following open conjecture. 

\begin{conj}[Problem 1.38 in \cite{Kirby1995ProblemsIL}]
    The Whitehead double of a knot $K$ is slice if and only if $K$ is slice. 
\end{conj}

Let $K$ be a knot and denote $\pmwhd$ as the positive ($+$) or negative ($-$) $t$-twisted Whitehead double of $K$. Specific details of the construction of Whitehead doubles are provided in Section \ref{section:background}.

\begin{thm}\label{prop:no4gWhd}
    $\gamma_{4} ( \pmwhd) \leq 2$.
\end{thm}

Note that this is a slight improvement from the existing bound given from the orientable 4-genus. We additionally show that non-orientable sliceness depends on the twisting parameter $t$. 

\newtheorem*{thm_1}{Theorem~\ref{thm:whdalternating4genus}}
\begin{thm_1} For any knot $K$, denote $\lambda$ as the Seifert framing, then
\begin{enumerate}[(i)]
    \item $\gamma_{4}(\text{Wh}^{+}_{t}(K) )= 2$ for $t>0$ when $t + \lambda $ is odd.
    \item $\gamma_{4}(\text{Wh}^{-}_{t}(K)) = 2$ for $t<0$ when $t + \lambda $ is odd.
\end{enumerate}
\end{thm_1}

This paper is structured as follows, section \ref{section:background} is an overview of slice knot obstructions and discusses the background of the slice problem, including obstructions to a knot being slice using invariants coming from knot Floer homology and 4-manifold theory. Section \ref{section:NOcobords} outlines the background for non-orientable 4 genus of knots, non-orientable 4 genus bounds for Whitehead doubles, and non-orientable cobordisms from Whitehead doubles to cable knots.

\textbf{Acknowledgments.} The author was partially supported by NSF Grant No. DMS-1907654 and NSF RTG Grant No. DMS-2231492. We thank Shelly Harvey for incredible feedback, and Josh Sabloff for pointing out a small computational error in the first draft of this paper.  

\section{THE SLICE PROBLEM}\label{section:background}

Artin asked a question about knots being represented as \textit{slices} of spheres in the 1920s, and this was the beginning of the study of knot concordance and 4 genus of knots. It seems to be a simpler task to show a knot is \textit{not} slice rather than to prove that it is, thus the use of obstructions is a popular method. There are many obstructions to a knot being slice that come from knot invariants. We will review classical knot invariants, Heegaard Floer invariants, and some 4-manifold theory. We begin by defining satellite knots and Whitehead doubles. For this paper, all knots $K$ are assumed to be in $S^{3}$ and all embeddings are assumed to be smooth, unless otherwise stated. The notational conventions used for knots $K$ in this paper come from the Rolfsen knot table, where $K$ is denoted $C_{n}$ where $C$ is the minimum number of crossings of any diagram of $K$, and $n$ is the index of the knot in the list of knots, per Rolfsen's knot table \cite{rolfsen}. 

\subsection{Whitehead Double}

A satellite knot is constructed with a \textit{pattern knot} $P$ in $S^{1} \times D^{2}$, and a \textit{companion knot} $K$ in $S^{3}$. See Figure \ref{fig:whdTpattern} for the $t$-twisted Whitehead double pattern. We construct the satellite knot $P(K)$ by removing a tubular neighborhood of $K$, denoted $n(K)$, and gluing in the solid torus containing $P$ where the longitude of $S^{1} \times D^{2}$ is identified with the Seifert framing of $K$. Note that it is common in literature for the neighborhood of a knot to be denoted $\nu(K)$, but as we will be discussing the $\nu$ invariant from Heegaard Floer homology in a later section, we use the notation $n(K)$ to avoid confusion. 

\begin{figure}[h]
    \centering
    \includegraphics[width=1.8in]{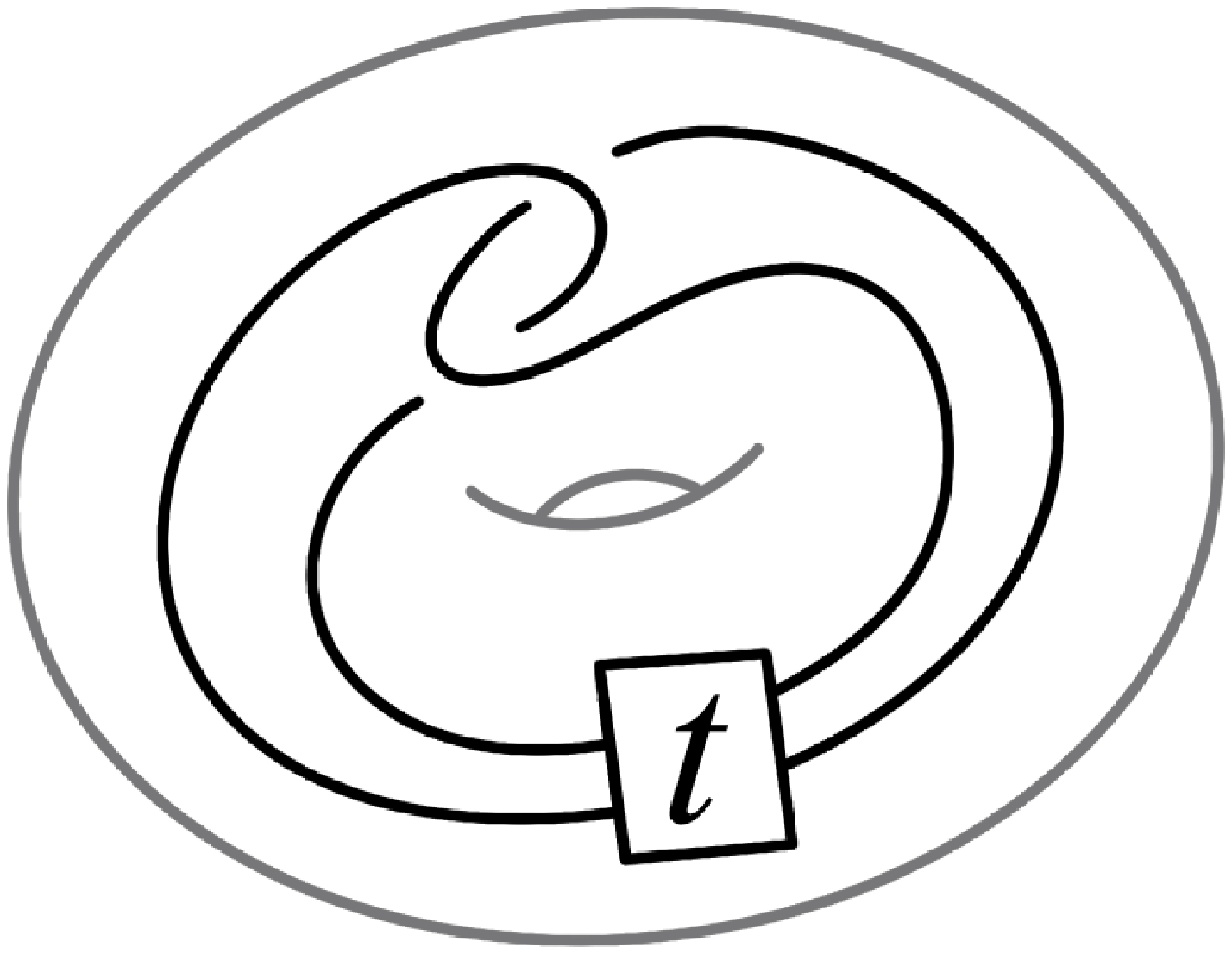}
    \caption{The Whitehead Double Pattern}
    \label{fig:whdTpattern}
\end{figure}

To see an example, we construct the knot of interest, the positive 0-twisted Whitehead double of the figure 8 knot in Figure \ref{fig:thewhdfig8knotconstruction}. Whitehead doubles have been a subject of interest for quite some time, for example the zero twisted Whitehead double of the trefoil knot was shown to not be smoothly slice by Akbulut in 1980 during a Santa Barbara Gauge Theory Conference, and the proof is given in Akbulut's book \cite{akbulutbook}.

\begin{figure}[h]
  \begin{subfigure}{0.31\textwidth}
    \includegraphics[width=\linewidth]{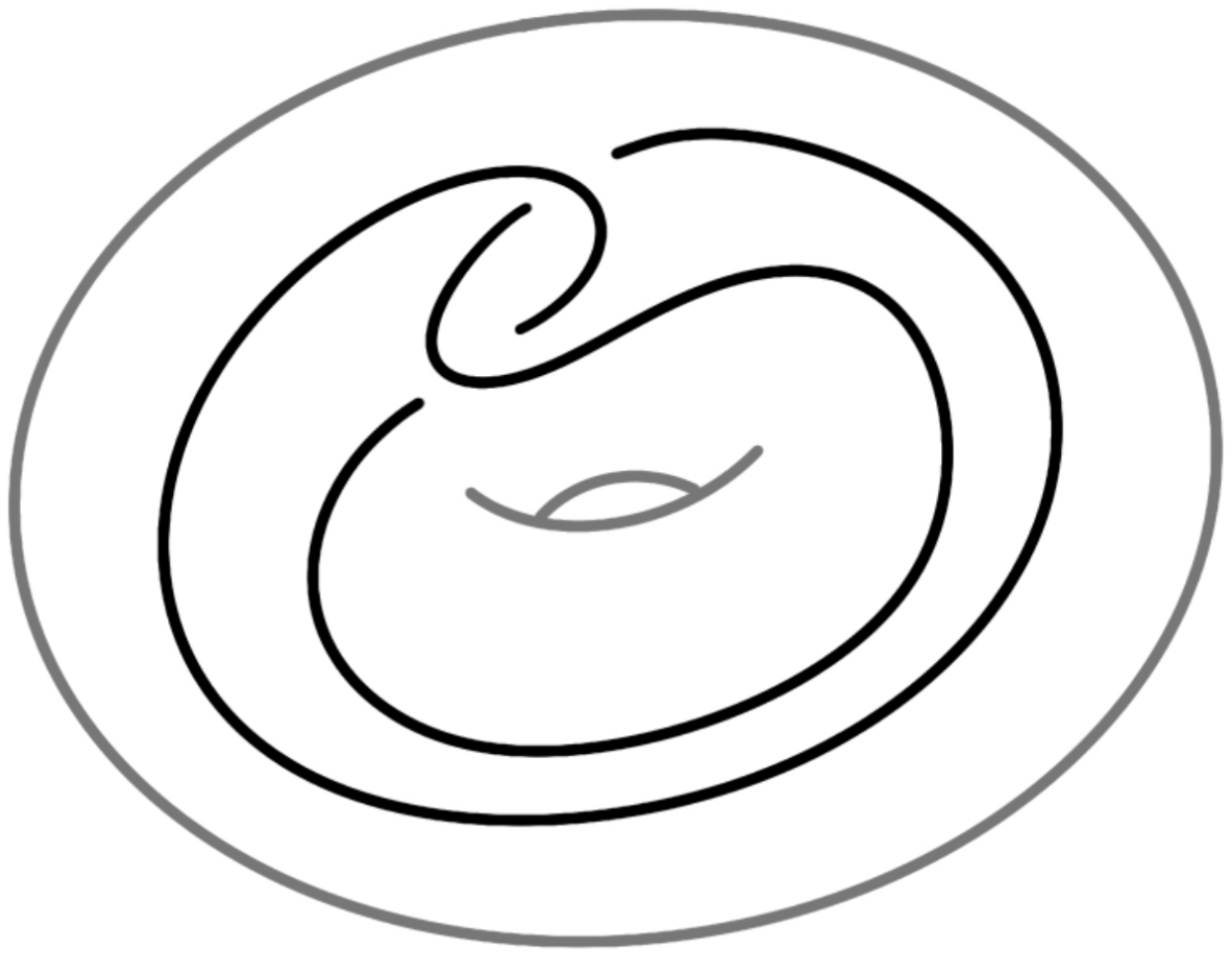}
    \caption{Untwisted Whitehead Double} 
  \end{subfigure}%
  \hspace{3mm} 
  \begin{subfigure}{0.22\textwidth}
    \includegraphics[width=\linewidth]{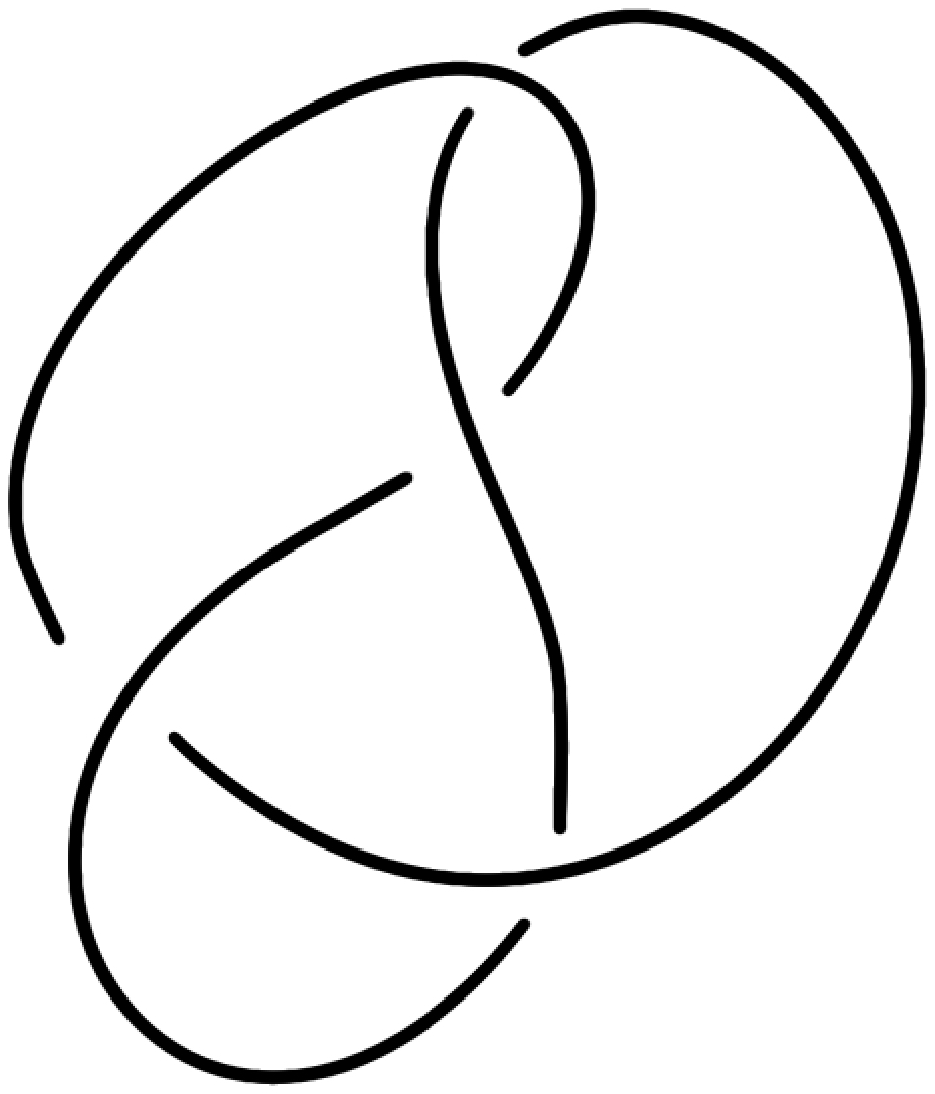}
    \caption{The Figure 8 Knot} \label{fig:fig8}
  \end{subfigure}%
  \hspace{3mm} 
  \begin{subfigure}{0.22\textwidth}
    \includegraphics[width=\linewidth]{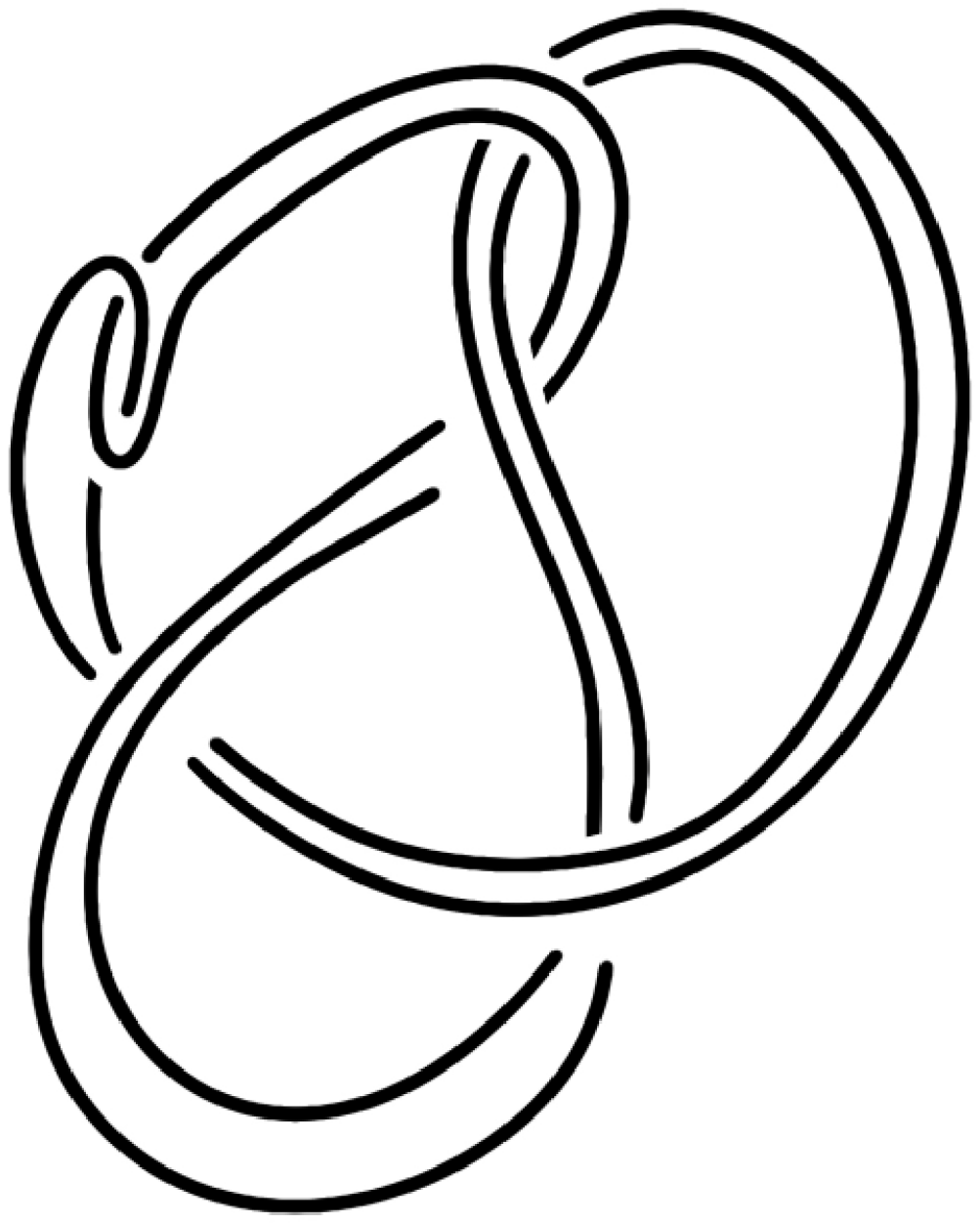}
    \caption{$\text{Wh}^{+}(4_{1})$} 
  \end{subfigure}

\caption{} \label{fig:thewhdfig8knotconstruction}
\end{figure}

\subsection{Classical Knot Invariants}

We begin with classical knot invariants. The invariants discussed in this paper are discussed in detail by Livingston \cite{livingstonKnotconc}. These invariants can all be calculated from a knot's \textit{Seifert surface}, which is an orientable surface bounded by a knot. We then define the orientable 3 genus, $g_{3}(K)$, as the minimum genus of a Seifert surface bounded by $K$. Most classical knot invariants for knots of 13 or less crossings have been calculated and are available on Knot Info \cite{knotinfo}. 

\begin{rem}
    It is a well known fact that $g_{4}(K) \leq g_{3}(K)$. One could see this to be true as imagining the Seifert surface bounded by $K$ being \textit{pushed} into the 4-ball. It is rare for equality to occur, an example is the unknot, which bounds a disk in $S^{3}$, so $g_{3}(U) = g_{4}(U) = 0$. 
\end{rem}

Given a knot $K$ and its Seifert surface $\Sigma$, we may construct the Seifert matrix, which is the matrix representing the Seifert pairing $V: H_{1}(\Sigma) \times H_{1}(\Sigma) \to \Z$. The map $V(x, y)$ is given by the linking number of $x$ and a positive pushoff of $y$. Using the Seifert matrix, we may compute many classical knot invariants, such as the signature, Arf invariant, and Alexander polynomial.  

The signature of a knot is denoted $\sigma(K)$, and is also referred to as the Murasugi signature \cite{signature}, and is calculated by $\sigma (V + V^{T})$ where $V$ is a Seifert matrix for $K$. 

\begin{thm}[Theorem 8.3 in \cite{signature}]
    If $K$ is slice, then $\sigma(K) = 0$.
\end{thm}

In fact, the signature offers a lower bound for the 4 genus. 

\begin{equation}
    \dfrac{|\sigma (K) |}{2} \leq g_{4}(K)
\end{equation}

The Arf invariant of a knot, $\arf(K)$, is either 0 or 1, and is calculated from the Seifert matrix $V$ by defining a non-singular quadratic form $q(x) = xVx^{T}$ and examining the Arf invariant of the quadratic form $q$, details can be found in \cite{livingstonKnotconc}.  If $K$ is a slice knot, it is well known that $\arf(K)=0$. 

\begin{ex}
    The right-handed trefoil knot, denoted $3_{1}$, is not a slice knot. We can determine this by calculating its signature, $\sigma (3_{1}) = -2$. The figure 8 knot, $4_{1}$, has signature $\sigma(4_{1}) = 0$ and $\arf (4_{1}) = 1$, so the signature does not obstruct sliceness, but the Arf invariant does. 
\end{ex}

Using the Seifert matrix $V$ for a knot $K$, one may compute the Alexander polynomial $\Delta_{K}(t) = \det(V - tV^{T}) \in \Z [t, t^{-1}]$. We then have the obstruction from Fox and Milnor \cite{FoxMilnor}. 

\begin{thm}[Fox-Milnor Condition \cite{FoxMilnor}]
    If $K$ is a slice knot, then its Alexander polynomial can be written as $\Delta_{K}(t) = f(t)f(t^{-1})$ with integer coefficients.
\end{thm}

We also have from Freedman \cite{topSlice} that if $\Delta_{K}(t) = 1$, then $K$ is topologically slice. There are several known examples of knots that are topologically slice but not smoothly slice, however this result can be extremely difficult to prove. 

\begin{rem}\label{rem:topologicallySliceWhd}
    It is well known that the Alexander polynomial of any untwisted Whitehead double of a knot $K$ is trivial. Thus presenting an infinite family of knots that are topologically slice but not smoothly slice. For thoroughness, we state the formula for calculating the Alexander polynomial of $b$-twisted Whitehead doubles in general. We briefly switch notation from $t$-twisted to $b$-twisted Whitehead doubles to avoid a notational nightmare around the variable $t$ in the Alexander polynomial and the number of twists $b$ for the Whitehead double. 
    $$\Delta_{\text{Wh}_{b}^{\pm}(K)}(t) = -b \cdot t + (2b +1) - b \cdot t^{-1}$$
We see that the Alexander polynomial for the untwisted Whitehead double $\Delta_{\text{Wh}^{\pm}(K)}(t) = 1$ for any knot $K$.     
\end{rem}

\begin{rem}
There is a quick computation for the Arf invariant of a knot if the Alexander polynomial is known, given by Murasugi \cite{murasugiARF}. If $\Delta_{K}(-1) \equiv \pm 1 \pmod{8}$, then $\arf (K) = 0$, otherwise $\arf (K) = 1$.  
\end{rem}

\begin{rem}\label{rem:whdclassicalinv}
    The untwisted Whitehead double of the figure 8 knot, $K = \text{Wh}(4_{1})$, has vanishing classical knot invariants. From Remark \ref{rem:topologicallySliceWhd} we know $\Delta_{K}(t) = 1$, and thus $\arf (K) = 0$. We also have that the signature of the untwisted Whitehead double of $K$ is equal to 0 \cite{HeddenWHD}. Thus, none of the classical knot invariants offer an obstruction to $\text{Wh}(4_{1})$ being smoothly slice. We compute the Seifert matrix in general for the $t$-twisted Whitehead doubles in Section \ref{sect:NOslice}.
\end{rem}

We mention that, in this paper, we are concerned with smooth 4 genus and smooth concordance. However, one may also be interested in algebraic concordance, which stems from \textit{Seifert forms} abstractly, and can be viewed through the lens of knot theory by taking Seifert surfaces of knots from which we compute Seifert forms. The difference between a Seifert form and a Seifert matrix is that a Seifert matrix specifically comes from a Seifert surface for a knot, while a Seifert form is defined from an abstract Seifert pairing on a finitely generated free $\Z$-module $M$ as a bilinear form $V: M \times M \to \Z$ satisfying $V - V^{T}$ is unimodular. Levine \cite{levine1, levine2} is credited with classifying the higher dimensional knot concordance to give a surjective homomorphism $\phi: \mathcal{C} \to \Z^{\infty} \oplus \Z_{2}^{\infty} \oplus \Z_{4}^{\infty}$, where $\mathcal{C}$ is the smooth knot concordance group. 

One could examine other knot signatures, such as the Levine-Tristram signatures, denoted $\sigma_{\omega}(V)$ for a Seifert form $V$ and $\omega$ a complex number. This invariant is over viewed in Livingston's survey \cite{livingstonKnotconc} and credited to Levine and Tristram \cite{LTsignature}. Moving toward concordance, abstractly, two Seifert forms $V_{1}$ and $V_{2}$ are algebraically concordant if $V_{1} \oplus -V_{2} $ is metabolic.  

\begin{thm}
    If $K$ is slice and $\Sigma$ a Seifert surface for $K$, then the associated Seifert form is metabolic. 
\end{thm}

This is a well known fact, and can be found as Theorem 2.6 in Livingston's survey \cite{livingstonKnotconc}, Theorem 1.8 from Conway \cite{Conwayalg}, or an interested reader could read about metabolic forms from Levine \cite{metabolicFormsLevine}. 

We note that Seifert forms may be used to obstruct sliceness, but algebraic concordance is weaker, in general, than smooth concordance. Given $K$ is algebraically slice, there are cases where it can be shown to not be smoothly slice using Casson-Gordon invariants \cite{cassongordon}. We will not define nor use Casson-Gordon invariants in this paper, but again refer to Livingston's survey \cite{livingstonKnotconc} and notes by Conway \cite{Conwayalg}.

\subsection{Knot Floer Homology}

There are several invariants coming from knot Floer homology that offer slice genus bounds. We provide the references for the bounds, and refer the reader to Ozsv\'ath and Szab\'o \cite{hfkandfourballgenus} for details on the construction of the Floer complexes and the details on computing the invariants.  

\begin{cor}[Corollary 1.3 in \cite{hfkandfourballgenus}]
   Given a knot $K$, $|\tau (K) | \leq g_{4}(K)$.
\end{cor}

Hedden \cite{HeddenWHD} discusses the knot Floer complex for Whitehead doubles, as well as calculates the $\tau$ invariant. 

\begin{thm}[Theorem 1.4 in \cite{HeddenWHD}]\label{thm:tau}
    $$ \tau( \whd) =\begin{cases}
			 0 &\text{ for } t \geq 2 \tau(K)  \\
		   1 &\text{ if } t < 2 \tau(K)  
		\end{cases}$$
\end{thm}

\begin{cor}[Corollary 1.5 in \cite{HeddenWHD}]
    $\tau(\whd) \neq 0 $ if and only if $\tau (K) > 0$. Hence, if $\tau(K) >0$ then $\whd$ is not smoothly slice for every $i$.
\end{cor}

We additionally have a bound coming from the $\epsilon$ invariant, originally introduced by Hom \cite{epsilon}. 

\begin{prop}[Proposition 3.6 in \cite{epsilon}]
    If $K$ is slice, then $\epsilon(K) = 0$.
\end{prop}

Patwardhan and Xiao compute $\epsilon$ for generalized Mazur patterns in \cite{epsilonSatellites}, however we may restrict their formula for Whitehead doubles in particular.

\begin{thm}[Theorem 1.5 in \cite{epsilonSatellites}]\label{thm:epsilonHFK}
     $$ \epsilon( \whd) =\begin{cases}
			 0 &\text{ if } \tau(K) = \epsilon(K) = 0  \\
		   1 &\text{ otherwise } 
		\end{cases}$$
\end{thm}

We then move on to the $\nu$ invariants, introduced in \cite{ossrational}, where $\nu(K)$ is either $\tau(K)$ or $\tau(K) + 1$. Also, the invariants $\nu^{-}$ and $\nu^{+}$ were shown to be equal by Ozv\'ath and Szab\'o in \cite{OS1}. Then we have the following shown by Rasmussen \cite{rasmussen}.

\begin{prop}
    Given a knot $K$, $\nu^{+} = \nu^{-} \leq g_{4}(K)$. 
\end{prop}

Details about the relations between these concordance invariants can be found in a survey by Hom \cite{homsurvey}.

Introduced by Ozv\'ath, Stipsicz, and Szab\'o in \cite{OS1}, the Upsilon invariant is a piece-wise linear invariant in the variable $t \in [0, 2]$. Upsilon is indeed a concordance invariant, and provides a lower bound for the 4 genus of a knot $K$. 

\begin{thm}[Theorem 1.11 in \cite{OS1}]
    The invariants $\Upsilon_{K}(t)$ bound the slice genus of $K$, that is, for $0 \leq t \leq 1$, 
    $$ | \Upsilon_{K}(t)| \leq t \cdot g_{4}(K). $$
\end{thm}

We then have a formula for the Whitehead double of knots in general, provided by Feller, Park, and Ray in \cite{satelliteUpsilon}. 

\begin{prop}[Proposition 1.5 in \cite{satelliteUpsilon}]\label{upsilonBounds}
    $$ \Upsilon_{\text{Wh}^{+}_{t}(K)}(t) =\begin{cases}
			 0 &\text{ if } t \geq 2 \tau(K)  \\
		   -1 + |1-t| &\text{ if } t < 2 \tau(K)  
		\end{cases}$$

  $$ \Upsilon_{\text{Wh}^{-}_{t}(K)}(t) =\begin{cases}
			 0 &\text{ if } t \leq 2 \tau(K)  \\
		   1 - |1-t| &\text{ if } t > 2 \tau(K)  
		\end{cases}$$
\end{prop}

\begin{rem}
Solving whether or not the Whitehead double of the figure 8 knot is proving to be quite difficult, as from Remark \ref{rem:whdclassicalinv} we know the classical knot invariants vanish, and now we see the Floer invariants vanish as well. From Proposition \ref{upsilonBounds}, we have that $\Upsilon_{\text{Wh}(4_{1})}(t) \equiv 0$. From Theorem \ref{thm:epsilonHFK}, we have that $\epsilon (\text{Wh}(4_{1})) = 0$. From Theorem \ref{thm:tau}, we have $\tau (\text{Wh}(4_{1})) = 0$. 
\end{rem}

For an in-depth discussion on the Whitehead double pattern in concordance with a gauge theory approach, see Hedden and Paul \cite{instantons}.  

\subsection{A 4-Manifold Obstruction}

We begin this discussion with defining the \textit{knot trace} as the 4-manifold $\xrk$ obtained by attaching a 2-handle to $B^{4}$ along a knot $K$ with framing $r$. For a formal background and introduction to 4-manifolds, we refer to the standard text by Gompf and Stipsicz \cite{gompfstipsicz}. 

Originally discussed by Kirby and Melvin in \cite{KirbyMelvin}, and later re-stated and proved by Miller and Piccirillo \cite{milpic}, we have a slicing obstruction from 4-manifold theory. 

\begin{thm}[\cite{KirbyMelvin}, \cite{milpic}]
    $K$ is smoothly slice if and only if $X_{0}(K)$ smoothly embeds in $S^{4}$. 
\end{thm}

This theorem motivates obstructing knots from being smoothly slice by using diffeomorphisms of knot traces. Given a knot $K$, if one finds a diffeomorphism of $X_{r}(K)$ to $X_{r}(J)$, where $J$ is a different knot that is not slice, then one may conclude that $K$ is not slice. 

Define the shake genus $g_{sh}^{r}(K)$ to be the minimum genus of a surface $S$ smoothly embedded in $\xrk$ so that $S$ is a generator of $H_{2}(\xrk ; \Z) \cong \Z$, and note that $g_{sh}^{r}(K) \leq g_{4}(K)$.  Piccirillo in \cite{piccirillo} shows that the smooth 4 genus is not 0-trace invariant, that is, there exist knots $K$ where $g_{sh}^{0}(K) < g_{4}(K)$. Also, Piccirillo \cite{piccirillo} uses Rasmussen's $s$-invariant, presented by Rasmussen in \cite{rasmSinv}, and shows that the $s$-invariant is not a 0-trace invariant. We do not define the $s$-invariant, but note that it is a concordance invariant and offers a lower bound on the smooth 4-genus.

The Conway knot problem was similar to the Whitehead double of the figure 8 problem, where all the classical invariants vanish, and it has trivial Alexander polynomial, so it was known to be topologically slice but questioned whether or not to be smoothly slice. Then, these 4-manifold techniques were used to show the Conway knot is not smoothly slice by Piccirillo in \cite{conwayknot}. Constructing diffeomorphic knot traces is a difficult task, but possible using 4-manifold handle slides, such as Akbulut did in \cite{akbulut1}, annulus twisting by Abe, Jong, Omae, and Takeuchi \cite{annulustwist}, or Piccirillo's RGB link construction \cite{piccirillo}.

As all previously mentioned invariants vanish for the Whitehead double of the figure 8 knot, the 4-manifold strategy seems to be the last remaining attempt to show an obstruction to sliceness. We were hoping to find a new obstruction, coming from non-orientable surfaces, however no such obstruction presented itself for the untwisted Whitehead double.  

We conclude this section by mentioning that recently Dai, Kang, Mallick, Park and Stoffregen showed the $(2, 1)$-cable of the figure 8 knot is not smoothly slice by showing its double branched cover does not bound an equivariant ball \cite{fig8cable}. We will not go into the construction or techniques, but mention this result as it is relevant to this paper.   

\section{NON-ORIENTABLE COBORDISMS}\label{section:NOcobords}

Recall that the non-orientable 4 genus of a knot $K$ is defined as the minimum first Betti number among all non-orientable surfaces $F$ smoothly embedded in $B^{4}$ bounded by $K$,
$$\gamma_{4}(K) = \{ \min b_{1}(F) \hspace{1mm} | \hspace{1mm} F \text{ is non-orientable, } \partial F = K, F \subset B^{4}  \}. $$

The property of non-orientable sliceness is when $\gamma_{4}(K) = 1$, a knot in $S^{3}$ bounds a smoothly embedded M\"obius band in $B^{4}$. We explore the upper bound from the 4-genus:
\begin{equation}\label{eqtn:4genusupperbound}
    \gamma_{4}(K) \leq 2 g_{4}(K) +1
\end{equation}
Given that a knot bounds some orientable surface $S \subset B^{4}$, and one may simply construct the connect sum $S \# \rp$, we obtain the upper bound in equation \ref{eqtn:4genusupperbound}. Many knots do not meet this upper bound, for example torus knots of the form $T_{2, q}$ all bound M\"obius bands \cite{Allen}, however have $g_{4}(T_{2, q}) = (q-1)/2$. Note that the orientable 4 genus of torus knots is given by the Milnor conjecture, which was proved by Kronheimer and Mrowka in \cite{milnorconj}. The Milnor conjecture states that the smooth 4 genus of a torus knot $T_{p, q}$ is $\frac{1}{2}(p-1)(q-1)$. On the other extreme of equation \ref{eqtn:4genusupperbound}, there are knots such as the $8_{18}$ knot in the Rolfsen knot table where $g_{4}(8_{18}) = 1$ and $\gamma_{4}(8_{18}) = 3 = 2g_{4}(8_{18})+1$. Thus, equation \ref{eqtn:4genusupperbound} is not a particularly useful upper bound for those attempting calculations of non-orientable 4 genus, however we have hope that it could offer an obstruction to sliceness. 

\begin{rem}
    If a knot $K$ does not bound a smoothly embedded M\"obius band, then $K$ is not smoothly slice. 
\end{rem}

We now define a \textit{cobordism} between two knots $K$ and $J$ as a surface $S$ in $S^{3} \times [0, 1]$ with two boundary components, $K$ in $S^{3} \times \{ 0\}$ and $-J$ in $S^{3} \times \{1\}$. We make quick remarks about orientation, as the convention is to say $K$ is cobordant to $J$, but the cobordism itself starts at $K$ in $S^{3} \times \{0\}$ and ends with $-J$ in $S^{3} \times \{1\}$, where $-J$ is the concordance inverse of $J$, which is the mirror and orientation reverse of $J$. $S^{3} \times \{0\}$ inherits the standard orientation from $S^{3}$, while $S^{3} \times \{1\}$ has the opposite orientation. The surface $S$ in $S^{3} \times [0, 1]$ may or may not be orientable. 

It is common to hear of a cobordism referred to as a surface interpolating between two knots. This surface has some genus, $g$, and if the surface is orientable with genus 0 it is called a \textit{concordance}. Through this lens, we say slice knots are those that are concordant to the unknot. If we allow the surface to be non-orientable, we have a \textit{non-orientable cobordism} between 2 knots. The genus 1 non-orientable cobordisms can be found directly by using \textit{non-orientable band moves}, as outlined below.

To perform a non-orientable band move, first one must orient a knot $K$, then attach an oriented band $B = [0, 1] \times [0, 1]$, where $B$ inherits its orientation from its boundary, to $K$ in such a way that $[0, 1] \times \{0\}$ agrees with the orientation of $K$ and $[0, 1] \times \{1\}$ disagrees with the orientation of $K$ (or vise-versa). One then performs band surgery, removing the arcs $[0,1] \times \{ 0 \}$ and $[0, 1] \times \{1 \}$. This process is depicted in Figure \ref{fig:bandmoves}. Note that after performing a non-orientable band move, we end up with an unoriented knot. This is in contrast with orientable band moves, which transform a knot into a link.

\begin{figure}[h]
    \centering
    \includegraphics[width=0.7\linewidth]{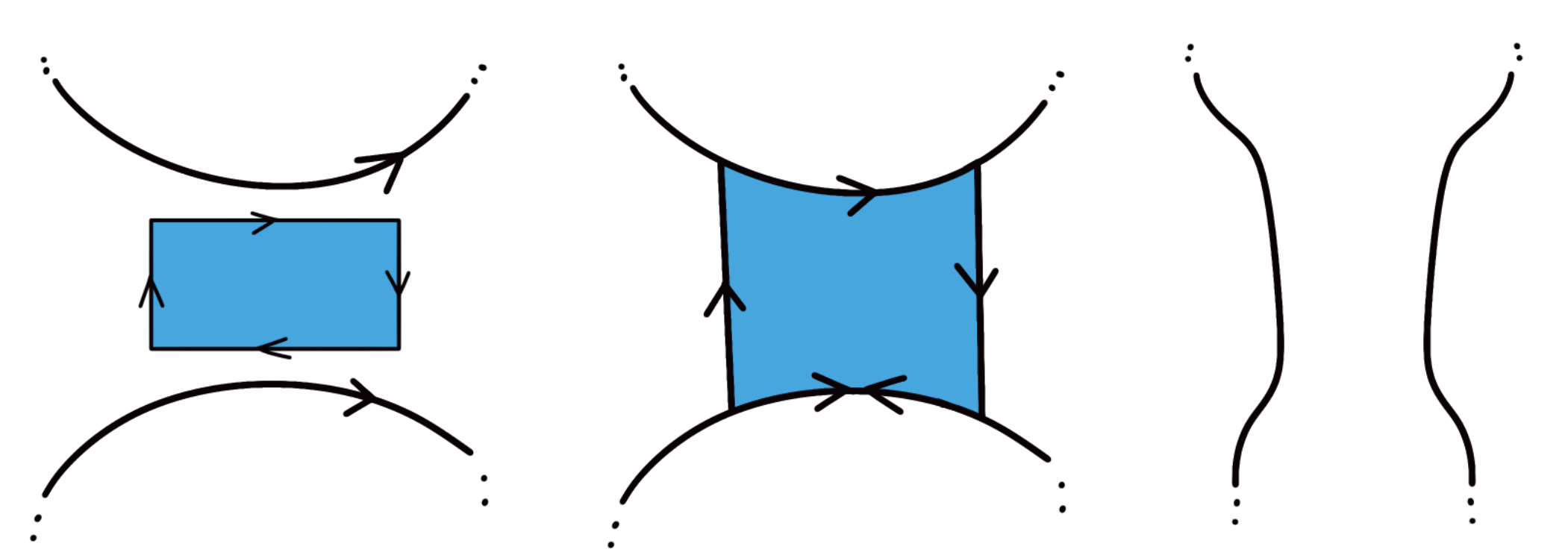}
    \caption{Non-Orientable Band Move}
    \label{fig:bandmoves}
\end{figure}

We then have the following proposition from Jabuka and Kelly. 

\begin{prop}[Proposition 2.4 in \cite{JK}]\label{BMbound}

If the knots $K$ and $J$ are related by a non-oriented band move, then 
$$ \gamma_{4}(K) \leq \gamma_{4}(J) + 1 $$

If a knot $K$ is related to a slice knot $J$ by a non-oriented band move, then $\gamma_{4}(K) = 1$.
    
\end{prop} 

We conclude this section by providing an example of a non-orientable band move from the figure 8 knot to the trefoil, Figure \ref{fig:exbandmoves}. 

\begin{figure}[h]
    \centering
    \includegraphics[width=0.7\linewidth]{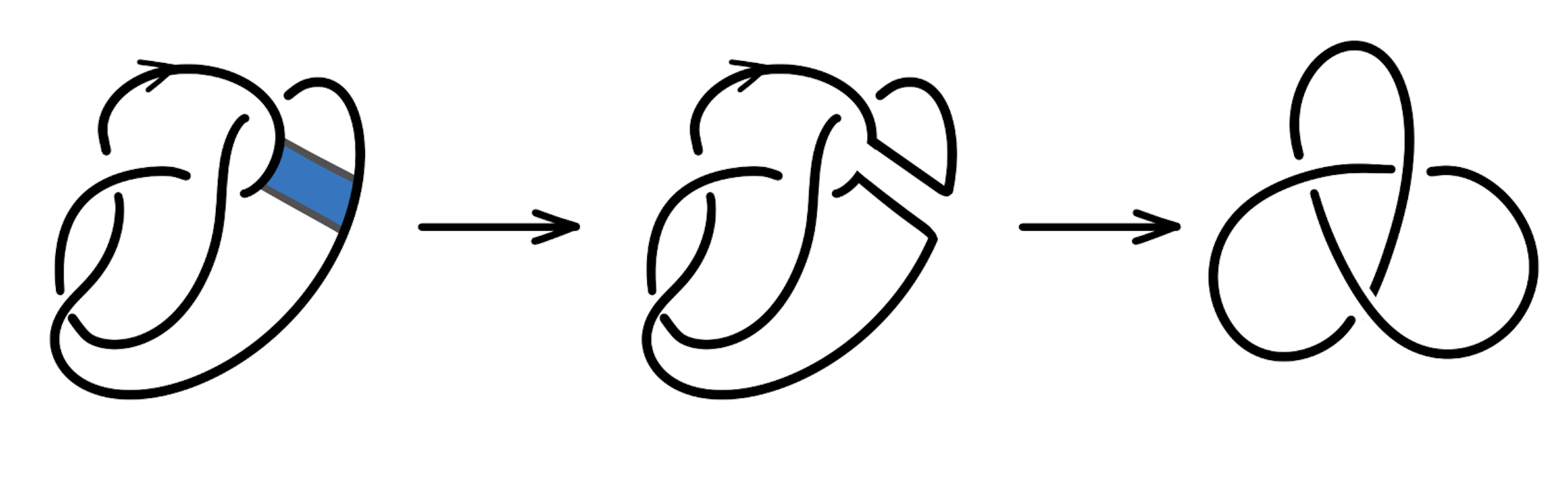}
    \caption{Non-Orientable Band Move from the figure 8 knot to the trefoil}
    \label{fig:exbandmoves}
\end{figure}

\subsection{Cables and more Knot Floer Homology}

 A knot cable is a specific type of satellite knot, where the pattern knot $P$ is a torus knot, $T_{p,q}$. When we construct a cable of a knot, it is called the $(p, q)$-cable of $K$ and denoted $K_{p, q}$. We again identify the longitude of the solid torus containing $P$ with the Seifert framing of $K$.

\begin{prop}
    Given a knot $K$, there exists an odd $q \in \Z$ so that there is a genus 1 non-orientable cobordism between $\text{Wh}^{\pm}_{t}(K)$ and $K_{2, q}$.  
\end{prop}

\begin{proof}
    Begin with an oriented knot $K$ and its Whitehead double. Performing a non-oriented band move on the clasp of the Whitehead double results in $K_{2, q}$ where $q$ depends on $t$ and the Seifert framing of $K$. See Figure \ref{fig:whdto21cable}. 
\end{proof}

\begin{figure}[h]
    \centering
    \includegraphics[width=4.5in]{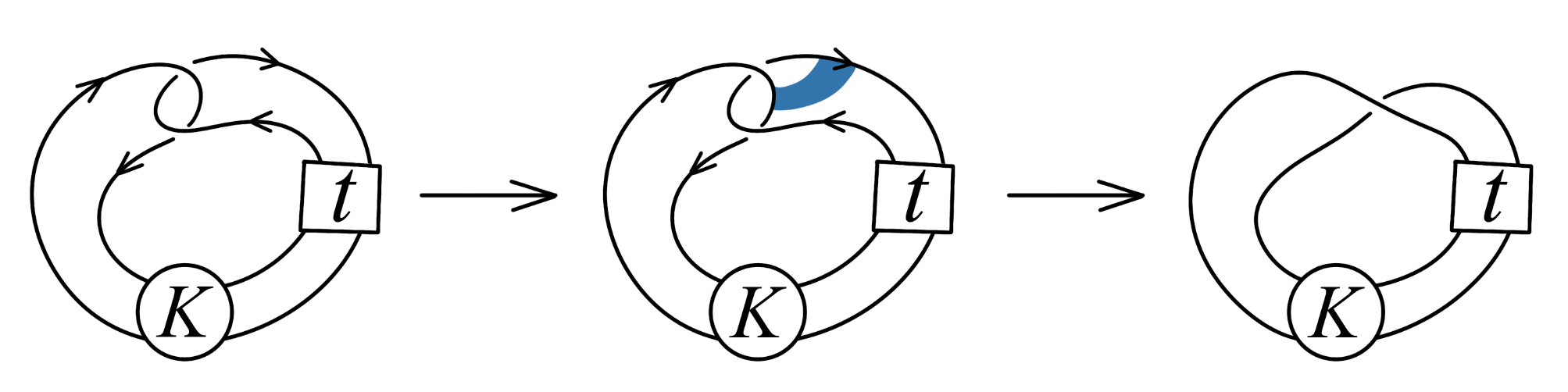}
    \caption{Non-Orientable Band Move}
    \label{fig:whdto21cable}
\end{figure}

To complete the proof of Proposition \ref{prop:no4gWhd} we notice that $\gamma_{4}(K_{2, q}) = 1$ for all $q$ and thus one non-orientable band move completes the notion that $\gamma_{4} ( \text{Wh}^{\pm}_{t}(K)) \leq 2$ according to Proposition \ref{BMbound}. For example, there is one non-orientable band move from $\text{Wh}^{+}(4_{1})$ to the $(2, 1)$-cable of the figure 8 knot, which was shown to not be smoothly slice in \cite{fig8cable}. 

Following the notational conventions of Ozv\'ath, Stipcisz, and Szab\'o in \cite{OSSnonorient}, we denote $\upsilon(K) = \Upsilon_{K}(1)$, the lower case Upsilon for when $t=1$. We now wish to examine consequences of the Upsilon invariant with this non-orientable cobordism from a Whitehead double to a cable. We re-state the proposition from Feller, Park, and Ray in \cite{satelliteUpsilon} for clarity when $t=1$.

\begin{prop}[Proposition 1.5 in \cite{satelliteUpsilon}]\label{prop:upsWHD}
    $$ \upsilon( \text{Wh}^{+}_{t}(K)) =\begin{cases}
			 0 &\text{ if } t \geq 2 \tau(K)  \\
		   -1 &\text{ if } t < 2 \tau(K)  
		\end{cases}$$

  $$ \upsilon( \text{Wh}^{-}_{t}(K)) =\begin{cases}
			 0 &\text{ if } t \leq 2 \tau(K)  \\
		   1 &\text{ if } t > 2 \tau(K)  
		\end{cases}$$
\end{prop}

\begin{thm}[Theorem 1.3 in \cite{upsCables}]
    Let $K$ in $S^{3}$ be a knot, and $(p, q)$ be a pair of relatively prime numbers such that $p > 0$. Then, 
    $$ \Upsilon_{K}(pt) - \dfrac{(p-1)(q+1)t}{2} \leq \Upsilon_{K_{p, q}}(t) \leq \Upsilon_{K}(pt) - \dfrac{(p-1)(q-1)t}{2} $$
    when $0 \leq t \leq \frac{2}{p}$.
\end{thm}

Thus, we have the following specialization for $t=1$.

\begin{cor}\label{cor:2qUpsilon}
    Given a knot $K$, $$ \left| \upsilon(K_{2, q}) + \dfrac{q}{2} \right| \leq 1$$ 
\end{cor}

Moving forward, we wish to discuss the \textit{normal Euler number} of a surface. Suppose $F$ is a non-orientable cobordism between knots $K$ and $K'$. First, define the \textit{local self-intersection number} of a surface $F$ by taking $F'$ to be a transverse push-off of $F$ that is compatible with the Seifert framings of $K$ and $K'$. The sign of a point $p \in F \cap F'$ is given by the choice of orientation of $T_{p}F$, the induced orientation of $T_{p}F'$, and the comparison of the orientation of $T_{p}F \oplus T_{p} F'$ with the orientation of $T_{p} ([0, 1] \times S^{3} )$. This comparison gives us a sign, $\pm 1$. Then, the \textit{normal Euler number} of a surface $F$, denoted $e(F)$, is the sum of local self-intersection numbers of points $p \in F \cap F'$ \cite{OSSnonorient}. Note that if we consider an orientable surface $S$, we have $e(S) = 0$. Similarly, $ e(F) = lk (K, s(K))$, where $s(K)$ is the non-vanishing section of the normal bundle $\nu_F$.

\begin{prop}
    Given a knot $K$, choose $q \in \Z$ so that $\pmwhd$ and $K_{2, q}$ cobound a genus 1 non-orientable surface $F$. Then, 
    $$ \left| \upsilon(\text{Wh}^{\pm}_{t}(K)) + \dfrac{q}{2} + \dfrac{e(F)}{4} \right| \leq \dfrac{3}{2} $$
\end{prop}

\begin{proof}
We first   Suppose that $F \subset [0, 1] \times S^{3}$ is a non-orientable genus 1 smooth cobordism from the knot $\pmwhd \subset \{0\} \times S^{3} $ to the knot $K_{2, q} \subset \{1\} \times S^{3}$. Then, as $b_{1}(F) = 1$, and using Theorem 1.1 from \cite{OSSnonorient} we have 
    $$\left| \upsilon(\pmwhd) - \upsilon(K_{2, q}) + \dfrac{e(F)}{4} \right| \leq \dfrac{1}{2}, $$
     Now from the triangle inequality and Corollary \ref{cor:2qUpsilon} we have 
    $$\left| \upsilon(K_{2, q}) + \dfrac{q}{2} + \upsilon(\pmwhd) - \upsilon(K_{2, q}) + \dfrac{e(F)}{4} \right| \leq \dfrac{3}{2}, $$
    thus concluding the proof. 
\end{proof}

As $\upsilon(\text{Wh}^{\pm}_{t}(K)) \in \{0, \pm 1\}$ this result gives us a relationship between $q$ and the normal Euler number of $F$, the genus one non-orientable surface co-bounded by $\text{Wh}^{\pm}_{t}(K)$ and $K_{2, q}$.

\section{NON-ORIENTABLE SLICENESS}\label{sect:NOslice}

The main goal of this section is to show that the non-orientable sliceness of Whitehead doubles is obstructed by $t$, the twisting parameter. 

\begin{thm}\label{thm:whdalternating4genus}
   For any knot $K$, denote $\lambda$ as the Seifert framing, then
\begin{enumerate}[(i)]
    \item $\gamma_{4}(\text{Wh}^{+}_{t}(K)) = 2$ for $t>0$ when $t + \lambda $ is odd.
    \item $\gamma_{4}(\text{Wh}^{-}_{t}(K)) = 2$ for $t<0$ when $t + \lambda $ is odd.
\end{enumerate} 
\end{thm}

We notice that this is quite different from the results on the orientable side, where Hedden \cite{HeddenWHD} showed the $\tau$ invariant for Whitehead doubles was non-zero over specific intervals. For the rest of this paper, it is assumed that when speaking of positive Whitehead doubles we assume $t > 0$, and negative Whitehead doubles we assume $t<0$. The reason for these assumptions is because when we consider $\text{Wh}^{+}_{t}(K)$ with $t$ negative, or $\text{Wh}^{-}_{t}(K)$ with $t$ positive, we end up with half-twisted Whitehead doubles, and these computational strategies fail when half-twists are involved. For example, $\text{Wh}^{+}_{-1}(K)$ gives the trefoil knot. See Figure \ref{fig:halftwistdbl} for an example of $\text{Wh}^{+}_{t}(U)$ for $t<0$. 

\begin{figure}[h]
    \centering
    \includegraphics[width=0.95\linewidth]{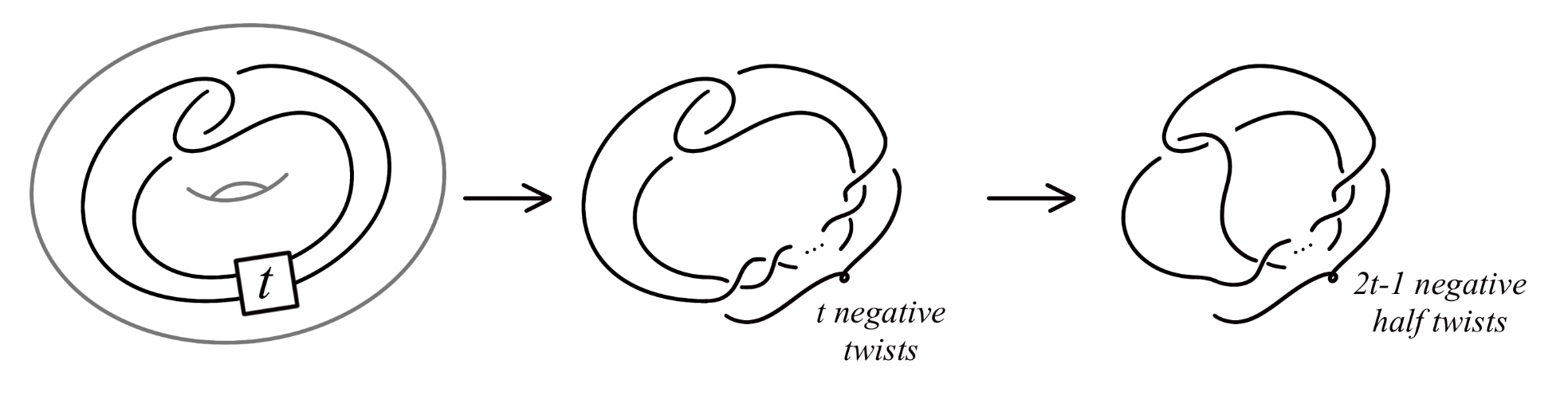}
    \caption{Constructing the half-twisted Whitehead double of the unknot}
    \label{fig:halftwistdbl}
\end{figure}

As previously mentioned, it is well known that $\sigma( \text{Wh}^{\pm}(K)) = 0$ for any knot $K$. We wish to expand this notion for the $t$-twisted Whitehead double of any knot, as well as providing a solution set for the Arf invariant. 

\begin{thm}\label{thm:whdsignature}
    For any knot $K$ in $S^{3}$, $\sigma (\pmwhd) = 0$. Additionally, denoting the Seifert framing of $K$ as $\lambda$, $\arf (\pmwhd) = 0$ if and only if $t+ \lambda $ is even. 
\end{thm}

We additionally discuss Theorem \ref{thm:NO4Gnotpreserved} in this section. We will do this by examining the Whitehead double in the presentation $K(a, b, c)$ presented in Livingston \cite{livingstonKnotconc}. Livingston discusses how when $a=1=c$, the knots $K(1, b, 1)$ are doubles of the unknot, see Figure \ref{fig:btwists}. We note that we switch the notation from Livingston and set $b=t$ and thus have $K(1, t, 1)$ for clarity in our work.

\begin{figure}
\begin{center}
  \begin{subfigure}{0.33\textwidth}
    \includegraphics[width=\linewidth]{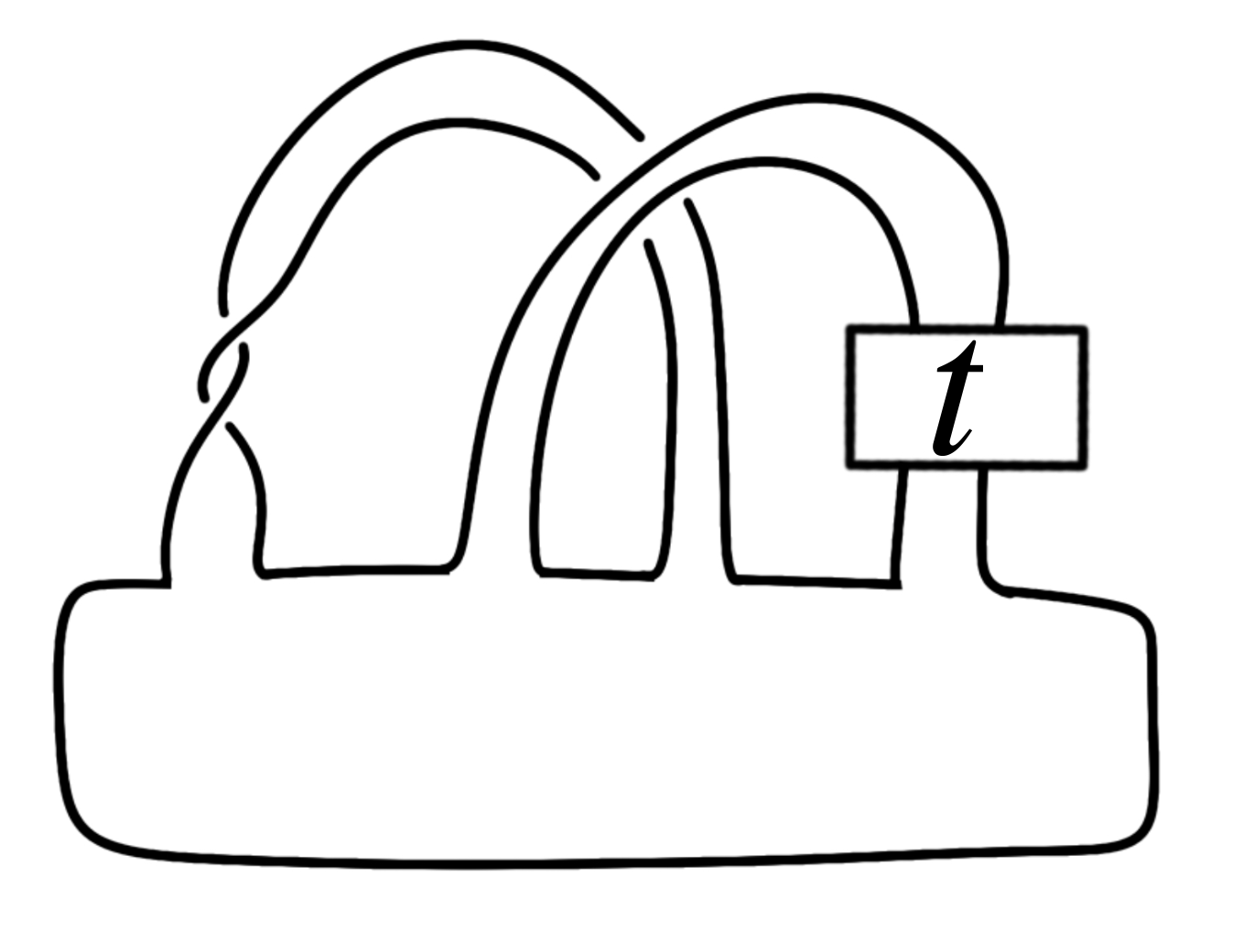}
    \caption{$K(1, t, 1)$} 
  \end{subfigure}%
  \hspace{3mm} 
  \begin{subfigure}{0.25\textwidth}
    \includegraphics[width=\linewidth]{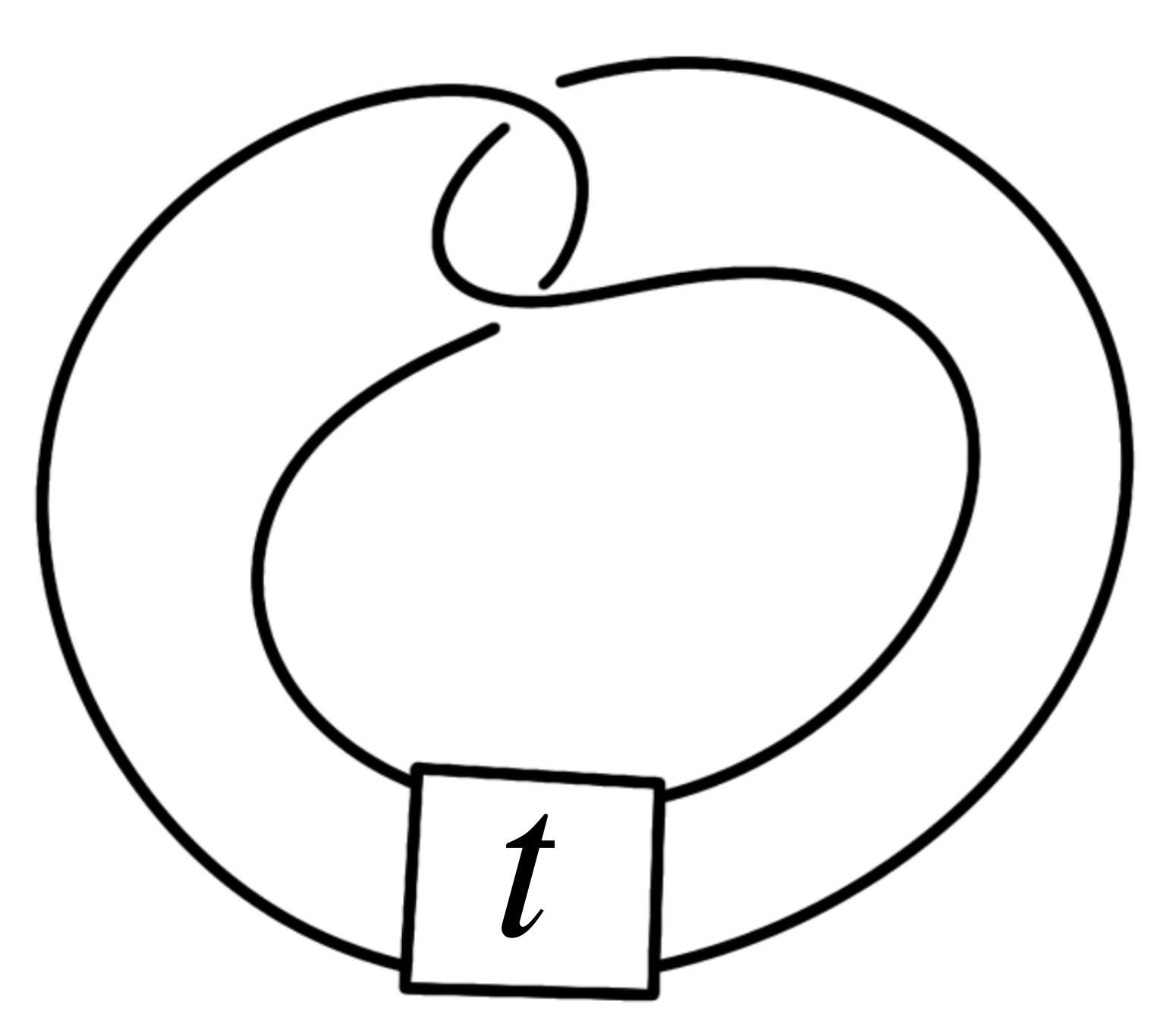}
    \caption{$\text{Wh}^{+}_{t}(U)$} 
  \end{subfigure}
  \end{center}

\caption{Isotopic Presentations of $K(1, t, 1) = \text{Wh}^{+}_{t}(U)$} \label{fig:btwists}
\end{figure}

\begin{proof}[Proof of Theorem \ref{thm:whdsignature}]

The proof will be computational, using the Seifert matrix for the $t$-twisted Whitehead double of the unknot to calculate the signature and Arf invariant for $t$-twisted Whitehead doubles of knots in general.  

To begin, compute the Seifert matrix $V$ for $K(1, t, 1)$, the positive $t$-twisted Whitehead double of the unknot. We see here that the Seifert matrix depends only on $t$. 
\[
V =
\begin{bmatrix}
    -1 & 1 \\
    0 & t
\end{bmatrix}
\]
Using $V$, we may compute the signature and Arf invariant. 

\begin{equation}\label{eqtn:sigmaB}
    \sigma(K(1, t, 1)) = 0 \text{ for all } t
\end{equation}

\begin{equation}\label{eqtn:ArfB}
    \arf  (K(1, t, 1)) = 0 \iff t \text{ is even}
\end{equation}

We note that the negative $t$-twisted Whitehead double has a similar Seifert matrix, differing only by a sign in one entry, and calculations of signature and Arf invariant yield the same conclusion. For the interested reader, the Seifert matrix $V_{-}$ in general for $K(-1, t, 1)$, the negative $t$-twisted Whitehead double of the unknot, is provided below.  
\[
V_{-}=
\begin{bmatrix}
    1 & 1 \\
    0 & t
\end{bmatrix}
\]

Thus, the Seifert matrix in general for $\pmwhd$ depends only on the sign of the clasp and the twisting parameter $t$. We generalize here to the Whitehead double of an arbitrary knot $K$, as all Whitehead doubles bound a genus one Seifert surface and will have Seifert matrices $V$ or $V_{-}$ that depend on $t$. That is to say, for an arbitrary knot $K$, the Seifert matrix corresponding to $\pmwhd$ will be identical to $V$ or $V_{-}$, respective to the positive/negative clasp. 

 We generalize Equation \ref{eqtn:sigmaB} to see $\sigma ( \pmwhd ) = 0$ for any knot $K$. Similarly, equation \ref{eqtn:ArfB} provides $\arf (\pmwhd ) = 0$ when $t$ is even. 
 
\end{proof}

Yasuhara \cite{yasuhara} gives us an obstruction to knots bounding a M\"obius band using these classical knot invariants.
\begin{prop}[Proposition 5.1 in \cite{yasuhara}]\label{sigArf}
    Given a knot $K$ in $S^{3}$, if $\sigma(K) + 4  \arf(K) \equiv 4 \pmod{8} $, then $\gamma_{4}(K) \geq 2$.
\end{prop}

\begin{thm}
    $\gamma_{4}(K(1, t, 1)) = 2$ when $t$ is odd. 
\end{thm}

\begin{proof}
    Let $t$ be an odd integer. This is a direct application of Yasuhara's obstruction in Proposition \ref{sigArf} to the computations in equations \ref{eqtn:sigmaB} and \ref{eqtn:ArfB} to see $\gamma_{4}(K(1, t, 1)) \geq 2$. As we have $\gamma_{4}(\pmwhd) \leq 2$ from Theorem \ref{prop:no4gWhd}, we may conclude $\gamma_{4}(K(1, t, 1)) = 2$, as desired. 
\end{proof}

As we have that the knots $K(1, t, 1)$ are doubles of the unknot, we directly conclude that the $t$-twisted Whitehead double takes the unknot, which has $\gamma_{4}(U)=1$, to a knot with non-orientable genus 2, thus not maintaining non-orientable sliceness, similar to the behavior of the pattern in the orientable setting. 

We also offer a different proof of Theorem \ref{prop:no4gWhd} using the \textit{crosscap number}, or non-orientable 3 genus of a knot $K$, $\gamma_{3}(K)$. Similar to the orientable case, it is well known that $\gamma_{4}(K) \leq \gamma_{3}(K)$. 

\begin{thm}
    Given $K$ is a knot in $S^{3}$, $\gamma_{3}(\pmwhd) \leq 2$.
\end{thm}

\begin{proof}
    We prove for the positive Whitehead double, however an identical method will work for the negative. Begin with the Whitehead pattern drawn in a solid torus. We then consider the genus of the pattern in the solid torus, which we define to be the minimum genus of a surface in the solid torus bounded by $P$, the pattern. The checkerboard coloring suffices to provide a punctured Klein bottle bounded by the untwisted positive Whitehead double, see Figure \ref{fig:thewhdsurfaces}. This gives us the desired upper bound for the crosscap number. 
\end{proof}

\begin{figure}[h]
\begin{center}
  \begin{subfigure}{0.25\textwidth}
    \includegraphics[width=\linewidth]{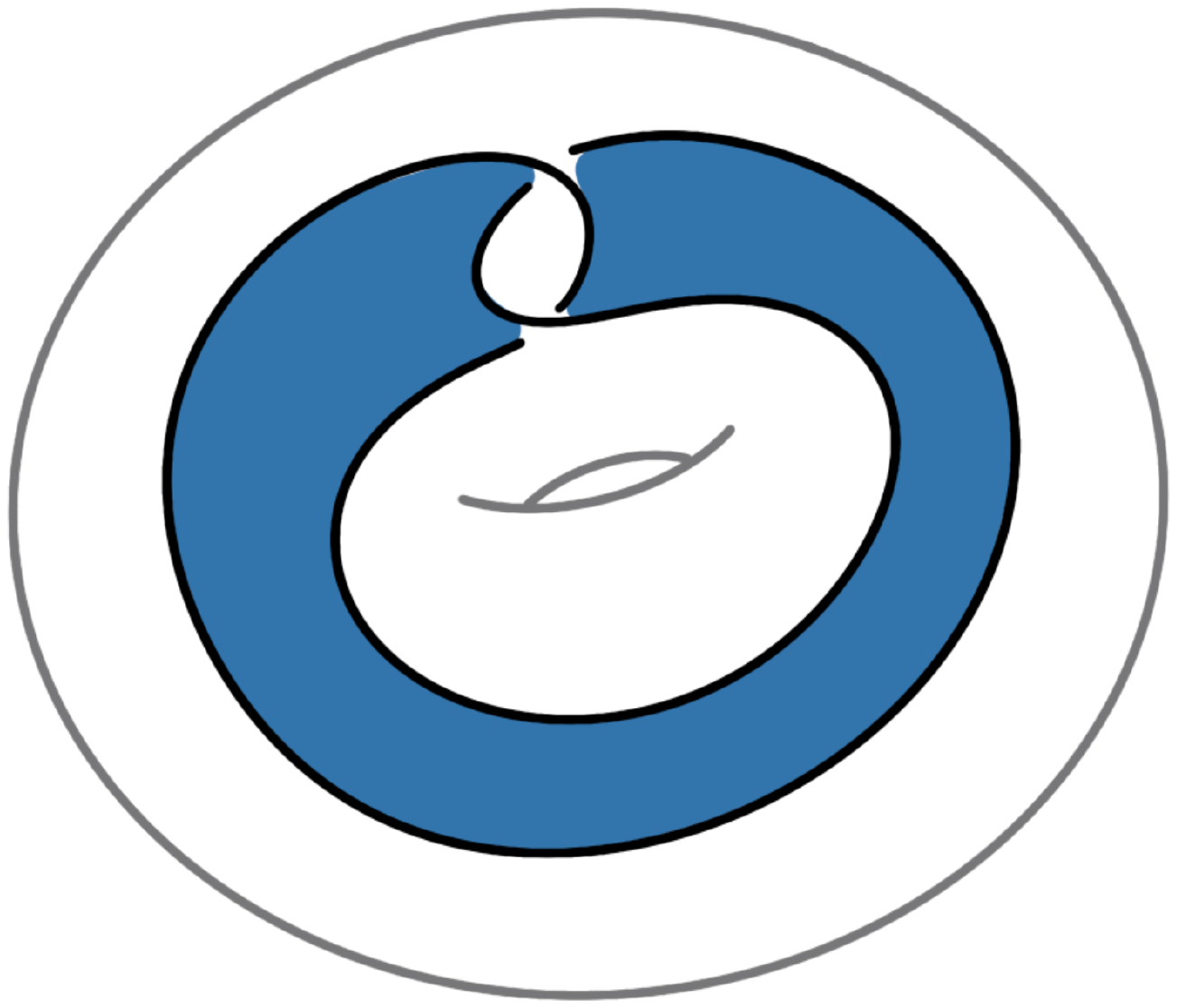}
    \caption{Untwisted Pattern} 
  \end{subfigure}%
  \hspace{3mm} 
  \begin{subfigure}{0.25\textwidth}
    \includegraphics[width=\linewidth]{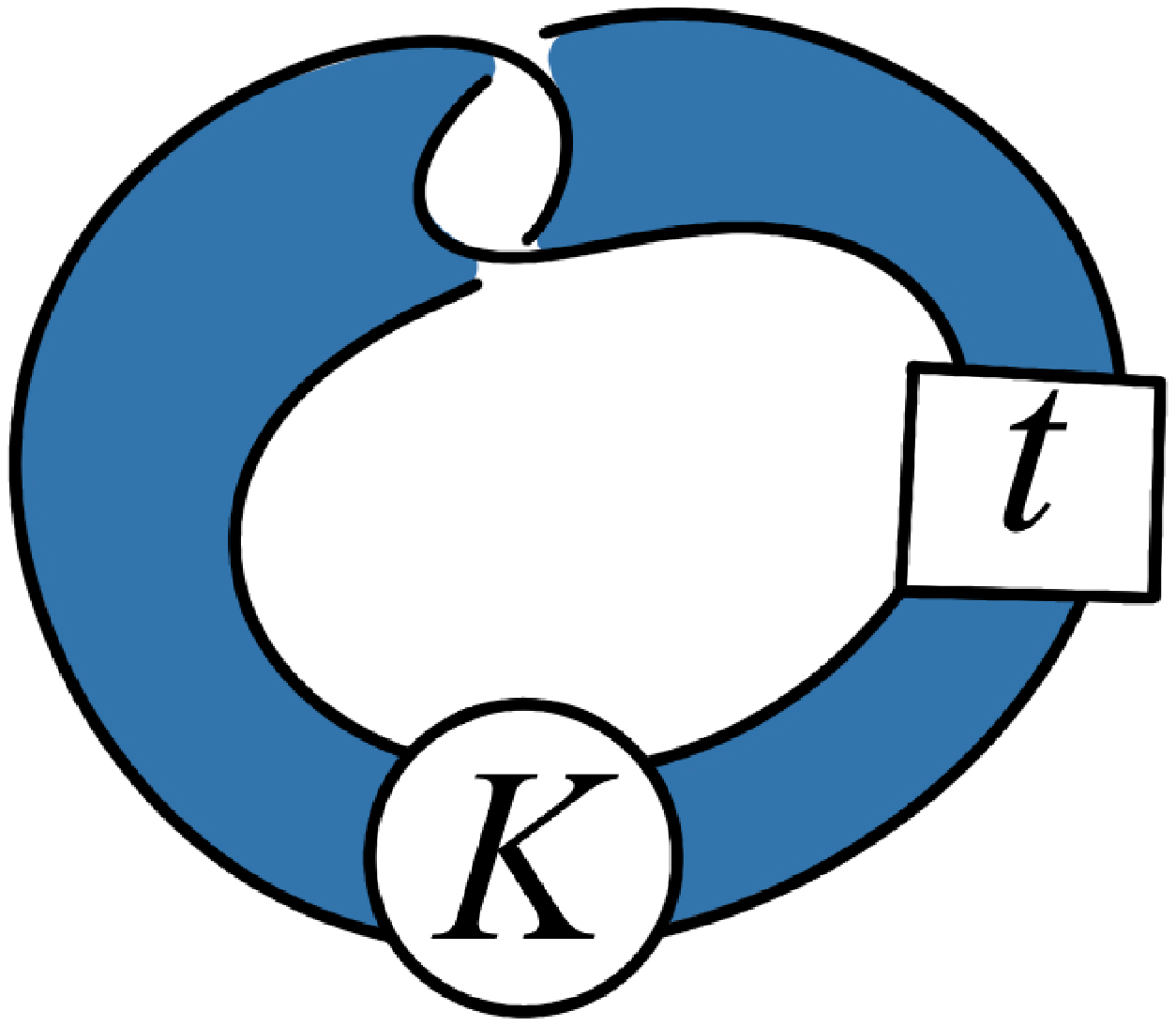}
    \caption{$\whd$} 
  \end{subfigure}
  \end{center}

\caption{Whitehead Double Bounding a Punctured Klein Bottle} \label{fig:thewhdsurfaces}
\end{figure}

We also wish to predict which knots $K$ have $\gamma_{4}(\whd) = 2$. From the signature and upsilon invariant, Ozv\'ath, Stipsicz, and Szab\'o \cite{OSSnonorient} provide the following lower bound on non-orientable 4 genus.

\begin{thm}[Theorem 1.2 in \cite{OSSnonorient}]\label{UpsSigBound}
  For a knot $K$ in $S^{3}$, 
    $$\left| \upsilon(K) + \frac{\sigma(K)}{2} \right| \leq \gamma_{4} (K)$$

\end{thm}

As $\sigma (\text{Wh}^{+}(K)) = 0$, the lower bound coming from Theorem \ref{UpsSigBound} depends only on $\tau (K)$. However, as $\upsilon(\pmwhd) \in \{0, \pm1\}$ for all $t$ any knot $K$, we see that the Upsilon invariant does not offer us an obstruction to Whitehead doubles bounding M\"obius bands.

\nocite{*}
\bibliography{main}

\bibliographystyle{plain}

\end{document}